\documentclass[reqno,b5paper]{amsart}

\usepackage{amsmath}
\usepackage{amssymb}
\usepackage{amsthm}
\usepackage{enumerate}
\usepackage[mathscr]{eucal}
\usepackage{eqlist}

\theoremstyle{plain}
\newtheorem{theorem}{Theorem}[section]
\newtheorem{lemma}{Lemma}[section]

\theoremstyle{definition}
\newtheorem{definition}{Definition}[section]
\newtheorem{remark}{\textnormal{\textbf{Remark}}}
\theoremstyle{remark}

\numberwithin{equation}{section}
\begin{document}

\title[3. Infinitesimal symmetries]%
{On the internal approach to differential equations\\ 3. Infinitesimal symmetries}
\author[Veronika Chrastinov\'a \and V\'aclav Tryhuk]%
{Veronika Chrastinov\'a \and V\'aclav Tryhuk}

\newcommand{\acr}{\newline\indent}

\address{Brno University of Technology\acr
Faculty of Civil Engineering\acr
Department of Mathematics\acr
Veve\v{r}\'{\i} 331/95, 602 00 Brno\acr
Czech Republic}

\email{chrastinova.v@fce.vutbr.cz, tryhuk.v@fce.vutbr.cz}

\thanks{This paper was elaborated with the financial support of the European
Union's "Operational Programme Research and Development for
Innovations", No. CZ.1.05/2.1.00/03.0097, as an activity of the
regional Centre AdMaS "Advanced Materials, Structures and
Technologies".}



\begin{abstract}
The geometrical theory of partial differential equations in the absolute sense, without any additional structures, is developed. In particular the symmetries need not preserve the hierarchy of independent and dependent variables. The order of derivatives can be changed and the article is devoted to the higher--order infinitesimal symmetries which provide a~simplifying "linear aproximation" of general groups of higher--order symmetries. The classical Lie's approach is appropriately adapted.
\end{abstract}

\maketitle

\section{Preface}\label{sec1}%%%%%%%%%%%%%%%%%%%%%%%%%%%%%%%%%%%%%%%%%%%%%%%%%%%%%%%
If the invertible higher--order transformations of differential equations are accepted as a~reasonable subject, the common Lie--Cartan's methods are insufficient for complete solution of the symmetry problem.
We recall that even the structure of all higher--order symmetries of the trivial (empty) systems of differential equations (that is, of the infinite--order jet spaces without any differential constraits) is unknown \cite {T1,T2,T3}. The same can be said for the "linearized theory" of the higher--order infinitesimal transformations treated in this article. 

Let us outline the core of the subject. We start with surfaces
\[w^1=w^1(x_1,\ldots,x_n), \ldots, w^m=w^m(x_1,\ldots,x_n)\]
lying in the space $\mathbb R^{m+n}$ with coordinates $x_1,\ldots,x_n,w^1,\ldots,w^m.$ The \emph{higher--order transformations} are defined by formulae
\begin{equation}\label{e1.1}\begin{array}{c} \bar x_i=W_i(\cdot\cdot,x_{i'},w^{j'}_I,\cdot\cdot), \bar w^j=W^j(\cdot\cdot,x_{i'},w^{j'}_I,\cdot\cdot)\\
(i,i'=1,\ldots,n;\, j,j'=1,\ldots,m)\end{array} \end{equation}
where the given smooth functions $W_i,W^j$ depend on the independent variables $x_1,\ldots,x_n$ and a~finite number of jet variables
\[w^j_I=\frac{\partial^{|I|}w^j}{\partial x_I}=\frac{\partial^{i_1+\cdots +i_r}w^j}{\partial x_{i_1}\cdots\partial x_{i_r}}\quad
(j=1,\ldots,m;\, i_1,\ldots,i_r=1,\ldots,n;\, r=0,1,\ldots\,).\]
The resulting surface
\[\bar w^1=\bar w^1(\bar x_1,\ldots,\bar x_n), \ldots, \bar w^m=\bar w^m(\bar x_1,\ldots,\bar x_n)\]
again lying in $\mathbb R^{m+n}$ appears as follows. We put
\begin{equation}\label{e1.2}
\bar x_i=W_i(\cdot\cdot,x_{i'},\frac{\partial^{|I|}w^{j'}}{\partial x_I}(x_1,\ldots,x_n),\cdot\cdot)=\bar x_i(x_1,\ldots,x_n)
\end{equation}
and assuming
\begin{equation}\label{e1.3}
\det\left(\frac{\partial \bar x_i}{\partial x_{i'}}\right)=\det\left(D_{i'}W_i\right)\neq 0\quad (D_i=\frac{\partial}{\partial x_i}+\sum w^j_{Ii}\frac{\partial}{\partial w^j_I}), \end{equation}
there exists the smooth inversion $x_i=x_i(\bar x_1,\ldots,\bar x_n)$ of the implicit function system (\ref{e1.2}) where $i=1,\ldots,n.$ This provides the result
\[\bar w^j=\bar w^j(\bar x_1,\ldots,\bar x_n)=W^j(\cdot\cdot,x_{i'}(\bar x_1,\ldots,\bar x_n),\frac{\partial^{|I|}w^{j'}}{\partial x_I}(\cdot\cdot,x_i(\bar x_1,\ldots,\bar x_n),\cdot\cdot),\cdot\cdot).\]
One can also obtain certain prolongation formulae
\[\bar w^j_I=\frac{\partial^{|I|}\bar w^j}{\partial\bar x_I}=W^j_I(\cdot\cdot,x_{i'},w^{j'}_{I'},\cdot\cdot)\quad (j,I \text{ as above})\]
for the derivatives by resolving the recurrence
\begin{equation}\label{e1.4}
\sum W^j_{Ii'}\,D_{i'}W_i=\sum D_iW^j_I.
\end{equation}
Functions $W_i$ satisfying (\ref{e1.3}) and $W^j$ may be arbitrary here.
It is however not easy to describe all \emph{invertible} transformations (\ref{e1.1}) and even more, to investigate the higher--order symmetries of differential equations.
So we recall the ancient infinitesimal version
\[\bar x_i=x_i+\varepsilon z_i(\cdot\cdot,x_{i'},w^{j'}_I,\cdot\cdot),\ \bar w^j=w^j+\varepsilon z^j(\cdot\cdot,x_{i'},w^{j'}_I,\cdot\cdot)\]
of formulae (\ref{e1.1}) with a~"small parameter $\varepsilon$". Then the invertibility mod $\varepsilon^2$ is trivially ensured by the change of $\varepsilon$ into $-\varepsilon.$ Alas, if we pass to rigorous exposition, quite other difficulties not occuring in the classical finite--dimensional theory appear.

Let us introduce the infinite--dimensional space $\mathbf M(m,n)$ with coordinates
\begin{equation}\label{e1.4a}
x_i,\ w^j_I\quad (j=1,\ldots,m;\, I=i_1\cdots i_r;\, i,i_1,\ldots,i_r=1,\ldots,n; r=0,1,\ldots \,)
\end{equation}
($I$ may be permuted) supplied with the module $\Omega(m,n)$ of \emph{contact forms}
\begin{equation}\label{e1.4b}
\omega =\sum a^j_I\omega^j_I\quad (\text{finite sum}, \omega^j_I=\mbox{d}w^j_I-\sum w^j_{Ii}\mbox{d}x_i).
\end{equation}
We are interested in the vector fields
\begin{equation}\label{e1.4c}
Z=\sum z_i\frac{\partial}{\partial x_i}+\sum z^j_I\frac{\partial}{\partial w^j_I}\quad (\text{infinite sum, arbitrary coefficients})
\end{equation}
such that $\mathcal L_Z\Omega(m,n)\subset\Omega(m,n)$ holds true for the Lie derivative $\mathcal L_Z.$
(Roughly saying, the contact forms are preserved after infinitesimal $Z$--shifts.) The inclusion is equivalent to the congruence
\[\mathcal L_Z\omega^j_I=Z\rfloor\mbox{d}\omega^j_I+\mbox{d}\,\omega^j_I(Z)=Z\rfloor\sum \mbox{d}x_i\wedge\omega^j_{Ii}+\mbox{d}\,\omega^j_I(Z)\cong 0\quad (\text{mod }\Omega(m,n))\]
which immediately gives the recurrence condition
\begin{equation}\label{e1.5} \omega^j_{Ii}(Z)=D_i\omega^j_I(Z)\quad \text{hence}\quad z^j_{Ii}=D_iz^j_I-\sum w^j_{Ii'}D_iz_{i'}\ ,
\end{equation}
this is the infinitesimal version of clumsy formulae (\ref{e1.4}).
With this preparation, we can eventually turn to the main topic.

The vector fields $Z$ satisfying $\mathcal L_Z\Omega(m,n)\subset\Omega(m,n)$ are called \emph{generalized} (or \emph{Lie--B\"acklund}) \emph{infinitesimal symmetries} of the jet space $\mathbf M(m,n)$ in actual literature.
However such $Z$ need not generate any Lie group which is in contradiction with the congenial classical point of view. So we prefer the shorter term \emph{variation} $Z$ in this case \cite{T4}. The \emph{infinitesimal symmetries} ensure the existence of a~\emph{true Lie group} in our conception.

An~infinitely prolonged system of differential equations
\begin{equation}\label{e1.5a}
D_{i_1}\cdots D_{i_r}f^k=0\quad (k=1,\ldots,K;\, i_1,\ldots,i_r=1,\ldots,n;\,r=0,1,\ldots\,)
\end{equation}
can be regarded as a~subspace $\mathbf M\subset\mathbf M(m,n).$ This is the \emph{external approach}, the reasonings are firmly localized in the ambient space $\mathbf M(m,n).$ Every~vector field~$Z$ tangent to $\mathbf M$ admits the natural restriction $Y$ to $\mathbf M.$ If $Z$ tangent to $\mathbf M$ is moreover a~variation in the ambient space, we speak of \emph{external variation} $Y$ of differential equations. 
 Let $\Omega$ be the restriction of module $\Omega(m,n)$ to $\mathbf M.$ If $Z$ is a~variation then clearly $\mathcal L_Y\Omega\subset\Omega$ and conversely, a~vector field $Y$ on $\mathbf M$ satisfying $\mathcal L_Y\Omega\subset\Omega$ can always be extended to a~variation~$Z$ (use the recurrence (\ref{e1.5}) restricted to $\mathbf M$).
So we may speak of \emph{internal variation}~$Y$ of differential equations as well and there is no essential distinction between external and internal concepts. 
Quite analogously, the infinitesimal symmetry $Z$ tangent moreover to $\mathbf M$ leads to \emph{external infinitesimal symmetry}~$Y$ of differential equations.
Let however $Y$ be a~vector field on~$\mathbf M$ such that $\mathcal L_Y\Omega\subset\Omega$ and let $Y$ \emph{generate a~Lie group on} $\mathbf M.$ Then we speak of \emph{internal infinitesimal symmetry}. Such $Y$ can always be extended into a~variation~$Z$ but this extension need not generate a~Lie group on the ambient space, i.e., $Y$ \emph{need not be} the external symmetry.

We are interested in the \emph{internal theory} in this article. For this aim, the space $\mathbf M$ equipped with module $\Omega$ will be characterized without any use of the localization in the ambient space $\mathbf M(m,n).$
%%%%%%%%%%%%%%%%%%%%%%%%%%%%%%%%%%%%%%%%%%%%%%%%
\section{Fundamental concepts}\label{sec2}
Though we develop \cite{T4,T5,T6}, the exposition is made selfcontained. All fundamental concepts are of the global nature, however, we deal with the local theory, that is, the definition domains are not discussed. No advanced technical tools are needed. We deal with modules of differential forms and vector fields together with the elementary algebra. The existence of bases of various modules to appear is tacitly postulated. A~certain novelty lies in the use of the infinite--dimensional manifolds, however, they are of a~classical nature without any functional analysis and norm estimates. 

Let $\mathbf M$ be a~smoth manifold \emph{modelled on} $\mathbb R^\infty,$ that is, there are \emph{coordinates} $h^i:\mathbf M\rightarrow\mathbb R$ $(i=1,2,\ldots\,)$ and the \emph{structural algebra} $\mathcal F$ of functions $f:\mathbf M\rightarrow\mathbb R$ expressible as the smooth composite $f=f(h^1,\ldots,h^{m(f)})$ in terms of coordinates. Then $\Phi$ denotes the $\mathcal F$--module of differential forms $\varphi=\sum f^i\mbox{d}g^i$ (finite sum with $f^i,g^i\in\mathcal F$) and $\mathcal T$ denotes the $\mathcal F$--module of vector fields $Z.$ They are regarded as $\mathcal F$--linear functions on $\mathcal F$--module $\Phi,$ i.e., we have $\mathcal F$--linear functions
\[\varphi(Z)=Z\rfloor\varphi=\sum f^i\text{d}g^i(Z)=\sum f^iZg^i\in\mathcal F\quad (Z\in\mathcal T \text{fixed})\]
of variable $\varphi\in\Phi.$

If $\varphi_1,\varphi_2,\ldots\,$ is a~basis of $\mathcal F$--module $\Phi,$ then the values $z^i=\varphi^i(Z)$ can be arbitrarily prescribed which is denoted by
\begin{equation}\label{e2.1}
Z=\sum z^j\frac{\partial}{\partial\varphi^j}\quad (\text{infinite sum}, z^j=\varphi^j(Z))\,.
\end{equation}
If in particular $\varphi^j=\mbox{d}h^j$ are differentials of coordinates, the well--known series
\[Z=(\sum z^j\frac{\partial}{\partial\mbox{d}h^j}=)\sum z^j\frac{\partial}{\partial h^j}\quad (\text{infinite sum}, z^j=Zh^j)\]
(abbreviation of notation) appears as a~particular subcase.
\begin{definition}\label{def2.1}
A~submodule $\Omega\subset\Phi$ of a~finite codimension $n=n(\Omega)$ is called a~\emph{diffiety} if there exists filtration $\Omega_*:\Omega_0\subset\Omega_1\subset\cdots\subset\Omega=\cup\Omega_l$ by finite--dimensional submodules $\Omega_l\subset\Omega$ $(l=0,1\ldots\,)$ satisfying
\begin{equation}\label{e2.2}
\mathcal L_\mathcal H\Omega_l\subset\Omega_{l+1}\quad (\text{all }l),\quad \Omega_l+\mathcal L_\mathcal H\Omega_l=\Omega_{l+1}\quad (l\text{ large enough}), \end{equation}
the so--called \emph{good filtration}. Here $\mathcal H=\mathcal H(\Omega)\subset\mathcal T$ is the submodule of all vector fields $Z$ such that $\Omega(Z)=0.$
\end{definition}
Diffieties $\Omega$ exactly correspond to the infinitely prolonged general systems of partial differential equations in the \emph{absolute sense}, i.e., without any additional structure. They realize the ancient E. Cartan's dream of the autonomous and coordinate--free theory in surprisingly simple and clear manner.

The technical concept of \emph{dependent variables} can be related to the choice of the filtration (\ref{e2.2}), see examples below and also the discussion in \cite{T4} for the particular case $n=n(\Omega)=1.$ The technical concept of \emph{independent variables} provides the link to the contact forms.
\begin{definition}\label{def2.2}
Functions $x_1,\ldots,x_n\in\mathcal F$ $(n=n(\Omega))$ are called \emph{independent variables} for the diffiety $\Omega$ if differentials $\mbox{d}x_1,\ldots,\mbox{d}x_n$ together with $\Omega$ generate the module $\Phi.$
\end{definition}
It follows that every form $\varphi\in\Phi$ admits a~unique representation
\begin{equation}\label{e2.3}
\varphi=\sum f^i\mbox{d}x_i+\omega\quad (f^i\in\mathcal F, \omega\in\Omega).
\end{equation}
The vector fields $D_1,\ldots,D_n\in\mathcal H$ uniquely defined by $\varphi(D_i)=f^i$ are called \emph{total derivatives} to the independent variables $x_1,\ldots,x_n.$ If in particular $\varphi=\mbox{d}f$ $(f\in\mathcal F),$ formula (\ref{e2.3}) provides the well--known \emph{contact forms}
\begin{equation}\label{e2.4}
\mbox{d}f-\sum f^i\mbox{d}x_i=\omega_f\in\Omega\quad (f^i=D_if=\mbox{d}f(D_i))
\end{equation}
of the diffiety $\Omega.$
\begin{definition}\label{def2.3}
A~vector field $Z\in\mathcal T$ is a~\emph{variation} of diffiety $\Omega$ if $\mathcal L_Z\Omega\subset\Omega.$ If $Z$ moreover generates a~Lie group, we speak of \emph{infinitesimal symmetry}.
\end{definition}
Vector fields will be represented by series (\ref{e2.1}) and appropriate choice of the forms $\varphi^j$ will simplify the calculation of variations due to the following lemma.
\begin{lemma}\label{l2.1}
A~vector field $Z\in\mathcal T$ is a~variation of diffiety $\Omega$ if and only if
\begin{equation}\label{e2.5}
(\mathcal L_{D_i}\omega)(Z)=D_i\omega(Z)\quad (i=1,\ldots,n;\,\omega\in\Omega).
\end{equation}
In fact only forms $\omega$ of a~basis of $\Omega$ are sufficient.
\end{lemma}
\begin{proof}
If $\omega\in\Omega$ then $\mbox{d}\omega\cong \sum \mbox{d}x_i\wedge\omega_i$ $(\text{mod }\Omega\wedge\Omega)$ 
for appropriate forms $\omega_i\in\Phi,$ however we infer that $\omega_i=\mathcal L_{D_i}\omega\in\Omega.$ If $Z\in\mathcal T$ is a~vector field, then
\[\mathcal L_Z\omega=Z\rfloor\mbox{d}\omega+\mbox{d}\,\omega(Z)\cong -\sum \omega_i(Z)\mbox{d}x_i+\sum D_i\omega(Z)\mbox{d}x_i\quad (\text{mod }\Omega)\]
by applying (\ref{e2.4}) to the function $f=\omega(Z)$ which implies (\ref{e2.5}). The last assertion of Lemma~\ref{l2.1} is trivial.
\end{proof}
Infinitesimal symmetries cause more difficulties. We can state the following general result \cite[Lemma 5.4, Theorems 5.6 and 11.1]{T7} without proof.
In examples to follow, simplified arguments will be enough.
\begin{lemma}\label{l2.2}
Let $\Gamma\subset\Omega$ be a~finite--dimensional submodule of diffiety~$\Omega$ such that $\Omega=\Gamma+\mathcal L_\mathcal H\Gamma+\mathcal L^2_\mathcal H\Gamma+\cdots$ and $Z$ be a~variation of~$\Omega.$ Then $Z$ generates a~Lie group (i.e., $Z$ is the infinitesimal symmetry of~$\Omega$) if and only if
\begin{equation}\label{e2.5a}
\mathcal L^{k+1}_Z\Gamma\subset \Gamma+\mathcal L_Z\Gamma+\cdots +\mathcal L_Z^k\Gamma
\end{equation}
 for appropriate $k$ large enough.
\end{lemma}
Recalling good filtration (\ref{e2.2}), one can choose $\Gamma=\Omega_l$ with $l$ large enough. 
We have introduced all fundamental concepts and technical tools for the subsequent exposition, however, three Remarks are still necessary for better clarity.

\begin{remark}\label{rem1}
Lemma~\ref{l2.1} replaces the vague condition $\mathcal L_Z\Omega\subset\Omega$ for the variation $Z$ with more effective condition (\ref{e2.5}). Though it is quite simple, the condition
\begin{equation}\label{e2.6} (\mathcal L_D\omega)(Z)=D\omega(Z)\quad (\omega\in\Omega,\, D\in\mathcal H)
\end{equation} clearly equivalent to (\ref{e2.5}) is still better. Paradoxically, it is only latently occuring in actual literature and we can refer to the ambitious exposition \cite[p. 107--113]{T8} which rests on rather special mechanism of "$\ell_\varphi$--linearization". The rule (\ref{e2.6}) involves this mechanism as a~particular subcase when $\omega=\omega_f\in\Omega$ is a~contact form and $Z=\frak Z$ is the \emph{evolutionary operator} satisfying the additional requirement $\frak Zx_i=0$ $(i=1,\ldots,n).$
It is also highly interesting to compare diffieties and symmetries \cite{T8} with our definitions \ref{def2.1}--\ref{def2.3}.
\end{remark}
\begin{remark}\label{rem2} The distinction between variations and infinitesimal symmetries is neglected in actual literature. For instance, clearly $\mathcal H\subset Var$ where $Var$ is the module of all variations, however, the factormodule $Sym=Var/\mathcal H$ introduced in \cite[pp. 9, 107]{T8} is a~confusing object in this respect since the class $[Z]\in Sym$ may involve both the true variations and the true infinitesimal symmetries. Moreover the frequent use of the evolutionary operator $\frak Z\in [Z]$ is not a~lucky measure since $\frak Z$ need not be the "best possible" element of the class $[Z]$ in the sense that only appropriate improvement $\frak Z+D\in [Z]$ $(D\in\mathcal H)$ may ensure a~true Lie group.
\end{remark}
\begin{remark}\label{rem3} If the underlying space $\mathbf M$ is of a~finite dimension, the above theory simplifies. The Frobenius theorem can be applied to the submodule $\Omega\subset\Phi$ and such a~diffiety $\Omega$ represents a~system of partial differential equations where the solution depends on a~finite number of constants.
\end{remark}
%%%%%%%%%%%%%%%%%%%%%%%%%%%%%%%%%%%%%%%%%%%%%%%%
\section{One independent variable}\label{sec3}
The particular case \mbox{$n=n(\Omega)=1$} of ordinary differential equations was treated in~\cite{T4}.
Then \emph{all} variations can be determined by a~mere linear algebra but no finite algorithm is as yet available for the totality of \emph{all} infinitesimal symmetries except the systems of $m$ equations with $m+1$ unknown functions at most. We discuss the differential equation
\begin{equation}\label{e3.1}\frac{d^2u}{dx^2}=F(x,u,v,\frac{du}{dx},\frac{dv}{dx},\frac{d^2v}{dx^2})\qquad (u=u(x), v=v(x))\end{equation}
as a~transparent example here.

Passing to the diffiety, let us introduce the space $\mathbf M$ with coordinates \[x,u_0,u_1,v_r\qquad (r=0,1,\ldots)\]
and diffiety $\Omega$ with the basis
\begin{equation}\label{e3.2}
\alpha_0=\mbox{d}u_0-u_1\mbox{d}x,\ \alpha_1=\mbox{d}u_1-F\mbox{d}x,\ \beta_r=\mbox{d}v_r-v_{r+1}\mbox{d}x\quad (r=0,1,\ldots\,)
\end{equation}
where $F=F(x,u_0,v_0,u_1,v_1,v_2).$ The total derivative
\[D=\frac{\partial}{\partial x}+u_1\frac{\partial}{\partial u_0}+F\frac{\partial}{\partial u_1}+\sum v_{r+1}\frac{\partial}{\partial v_r}\in\mathcal H\]
clearly satisfies $\mathcal L_D\alpha_0=\alpha_1$ and
\begin{equation}\label{e3.3}
\mathcal L_D\alpha_1=\mbox{d}F-DF\mbox{d}x=F_{u_0}\alpha_0+F_{v_0}\beta_0+F_{u_1}\alpha_1+F_{v_1}\beta_1+F_{v_2}\beta_2,
\end{equation}
moreover $\mathcal L_D\beta_r=\beta_{r+1}.$
The order--preserving filtration $\Omega_*$ consists of submodules $\Omega_l\subset\Omega$ $(l=0,1,\ldots\,)$ generated by the forms $\alpha_r,\beta_r$ where $r\leq l.$
%%%%%%%%%%%%%%%%%%
\subsection{The direct approach to variations}\label{ssec3.1}
Using (\ref{e2.5}) with $n=1$ and $D_1=D,$ we conclude that $Z$ is a~variation if and only if $\alpha_1(Z)=D\alpha_0(Z)$ and
\[D\alpha_1(Z)=F_{u_0}\alpha_0(Z)+F_{v_0}\beta_0(Z)+F_{u_1}\alpha_1(Z)+F_{v_1}\beta_1(Z)+F_{v_2}\beta_2(Z),\]
moreover $\beta_{r+1}(Z)=D\beta_r(Z).$ Assuming the development
\[Z=z\frac{\partial}{\partial x}+z^0\frac{\partial}{\partial u_0}+z^1\frac{\partial}{\partial u_1}+\sum z_r\frac{\partial}{\partial v_r}\quad (z=Zx,z^r=\alpha_r(Z),z_r=\beta_r(Z))\]
of the kind (\ref{e2.1}), we have the recurrences $z^1=Dz^0,z_{r+1}=Dz_r$ $(r=0,1,\ldots)$ together with the crucial requirement
\begin{equation}\label{e3.4}D^2z^0=F_{u_0}z^0+F_{v_0}z_0+F_{u_1}Dz^0+F_{v_1}Dz_0+F_{v_2}D^2z_0\end{equation}
for the initial coefficients $z_0$ and $z^0$ (coefficient $z$ is arbitrary). It is not easy to resolve equation (\ref{e3.4}) where the functions $z_0, z^0$ depend on finite but uncertain number of coordinates.
%%%%%%%%%%%%%%%%%%
\subsection{Better approach}\label{ssec3.2}
We introduce the alternative basis $\pi_r$ $(r=0,1,\ldots)$ of diffiety $\Omega$ such that $\mathcal L_D\pi_r=\pi_{r+1},$ the \emph{standard basis} \cite {T4}. Then $Z$ is a~variation if and only if $\pi_{r+1}(Z)=D\pi_r(Z)$ and therefore the formula
\begin{equation}\label{e3.5} Z=z\frac{\partial}{\partial x}+\sum D^rp\frac{\partial}{\partial\pi_r}\qquad (p=\pi_0(Z)\text{ hence }D^rp=\pi_r(Z))\end{equation}
with arbitrary functions $z$ and $p$ provides all variations $Z.$ (Some "degenerate cases" are omitted here, see below.)

In order to obtain the standard basis, identity (\ref{e3.3}) will be reduced. Clearly
\[\mathcal L_D(\alpha_1-F_{v_2}\beta_1)=F_{u_0}\alpha_0+F_{v_0}\beta_0+F_{u_1}\alpha_1+(F_{v_1}-DF_{v_2})\beta_1.\]
Denoting $\alpha=\alpha_1-F_{v_2}\beta_1,$ the form $\alpha_1$ can be replaced with $\alpha$ in the original basis (\ref{e3.2}) and identity (\ref{e3.3}) simplifies as
\[\mathcal L_D\alpha=F_{u_0}\alpha_0+F_{v_0}\beta_0+F_{u_1}\alpha+A\beta_1\quad (A=F_{v_1}+F_{u_1}F_{v_2}-DF_{v_2}).\]
The last summand can be deleted as well. Clearly
\[\mathcal L_D(\alpha-A\beta_0)=F_{u_0}\alpha_0+(F_{v_0}-DA)\beta_0+F_{u_1}\alpha,\]
\[\mathcal L_D\alpha_0=\alpha_1=\alpha+F_{v_2}\beta_1\quad\text{hence}\quad \mathcal L_D(\alpha_0-F_{v_2}\beta_0)=\alpha-DF_{v_2}\beta_0.\]
Denoting $\beta=\alpha-A\beta_0, \gamma=\alpha_0-F_{v_2}\beta_0,$ both forms $\alpha$ and $\alpha_0$ can be replaced with $\beta$ and $\gamma.$ We have the basis
\begin{equation}\label{e3.6}
\beta=\alpha-A\beta_0,\ \gamma=\alpha_0-F_{v_2}\beta_0,\ \beta_r\qquad (r=0,1,\ldots\,)\end{equation}
satisfying
\[\mathcal L_D\beta=F_{u_0}(\gamma+F_{v_2}\beta_0)+F_{u_1}(\beta+A\beta_0)+(F_{v_0}-DA)\beta_0=F_{u_0}\gamma+F_{u_1}\beta+B\beta_0,\]
\[\mathcal L_D\gamma=\alpha-DF_{v_2}\beta_0=\beta+A\beta_0-DF_{v_2}\beta_0=\beta+C\beta_0\]
with certain coefficients $B$ and $C.$ Finally
\[\mathcal L_D(C\beta-B\gamma)=DC\beta-DB\gamma+C\mathcal L_D\beta-B\mathcal L_D\gamma=M\beta+N\gamma\]
with certain $M,N.$ We are done.
%%%%%%%%%%%%%%%%%%
\subsection{The controllable subcase}\label{ssec3.3}
In general
\[\det\left(\begin{array}{rr}
C&-B\\
M&N
\end{array}\right)\neq 0.\]
Then we have the standard basis
\[\pi_0=C\beta-B\gamma,\ \pi_1=\mathcal L_D\pi_0=M\beta+N\gamma,\ \pi_r=\mathcal L_D^r\pi_0\quad (r=2,3,\ldots\,).\]
Indeed, the forms $\beta,\gamma$ are linear combinations of $\pi_0,\pi_1.$ Either $B\neq 0$ or $C\neq 0$ therefore $\beta_0$ can be expressed in terms of $\beta,\gamma,\mathcal L_D\beta,\mathcal L_D\gamma$ and therefore in terms of $\pi_0,\ldots,\pi_4.$ Then the forms $\beta_r=\mathcal L^r_D\beta_0$ are involved, too.
%%%%%%%%%%%%%%%%%%
\subsection{On the noncontrollable subcase}\label{ssec3.4}
If $B=C=0$ identically then $\text{d}\beta=\text{d}\gamma\cong 0$ (mod $\beta,\gamma$) and $F=D^2G$ for appropriate $G=G(x,u_0,v_0).$ If either $B\neq 0$ or $C\neq 0$ but $\det\left(\begin{array}{rr}
C&-B\\
M&N
\end{array}\right)= 0$
then the form $C\beta-B\gamma$ is a~multiple of a~total differential and $F=DG$ where $G=G(x,u_0,v_0,u_1,v_1).$ We refer to~\cite{T4} for this "degenerate" case.
%%%%%%%%%%%%%%%%%%
\subsection{Towards the infinitesimal symmetries}\label{ssec3.5}
While variations were obtained in full generality, the additional requirements of Lemma~\ref{l2.2} cannot be completely analyzed in full generality here due to a~limited space. We may refer to~\cite{T4} for the "simple" function $F=F(v_1).$ So we shall deal with analogous "easier" function.
%%%%%%%%%%%%%%%%%%
\subsection{Survey of explicit formulae}\label{ssec3.6}
Let us suppose $F=u_0v_1$ from now on. In the direct approach, equation (\ref{e3.4}) simplifies as
\begin{equation}\label{e3.7}
D^2z^0=v_1z^0+u_0Dz_0\quad (D=\frac{\partial}{\partial x}+u_1\frac{\partial}{\partial u_0}+u_0v_1\frac{\partial}{\partial u_1}+\sum v_{r+1}\frac{\partial}{\partial v_r}). \end{equation}
We shall soon state the explicit solution.

Turning to the standard basis, one can see that
\[\alpha=\alpha_1,\beta=\alpha_1-u_0\beta_0,\gamma=\alpha_0,\pi_0=u_0\beta+u_1\gamma,\pi_1=2u_1\beta+2u_0v_1\gamma.\]
This implies
\[\Delta\beta=2u_0v_1\pi_0-u_1\pi_1,\Delta\gamma=-2u_1\pi_0+u_0\pi_1\quad (\Delta=2u_0^2v_1-2u_1^2)\]
and therefore
\[\Delta\alpha_0=-2u_1\pi_0+u_0\pi_1, u_0\Delta\beta_0=\Delta\alpha_1-2u_0v_1\pi_0+u_1\pi_1.\]
So we have the solution
\begin{equation}\label{e3.8}
\begin{array}{l}
z^0=\alpha_0(Z)=\dfrac{1}{\Delta}(-2u_1p+u_0Dp),\\\\
z_0=\beta_0(Z)=\dfrac{1}{u_0}Dz^0-\dfrac{2v_1}{\Delta}p+\dfrac{u_1}{u_0\Delta}Dp\end{array}\end{equation}
of equation (\ref{e3.7}) with arbitrary function $p.$ As yet we have dealt with variations only.
%%%%%%%%%%%%%%%%%%
\subsection{The realm of true symmetries}\label{ssec3.7}
Let us consider variations $Z$ such that moreover
\begin{equation}\label{e3.9} \mathcal L_Z\pi_0=\lambda\pi_0 \end{equation}
for an~appropriate factor $\lambda\in\mathcal F.$ In more detail, we have the requirement
\[Z\rfloor\text{d}\pi_0+\text{d}p=\lambda\pi_0\quad (p=\pi_0(Z))\]
where
\[\text{d}\pi_0=\text{d}x\wedge\pi_1+\alpha_0\wedge\beta+u_0(-\alpha_0\wedge\beta_0)+\alpha_1\wedge\gamma=\text{d}x\wedge\pi_1+2u_0\beta_0\wedge\gamma,\]
\[\text{d}p=Dp\,\text{d}x+p_{u_0}\alpha_0+p_{u_1}\alpha_1+\sum p_{v_r}\beta_r.\]
We insert $\alpha_0=\gamma, \alpha_1=\beta+u_0\beta_0$ in order to use the advantageous intermediate basis (\ref{e3.6}). Then (\ref{e3.9}) is expressed by the identity
\[(2u_1z+p_{u_1})\beta+(2u_0v_1z+p_{u_0})\gamma+(-2u_0c+u_0p_{u_1}+p_{u_0})\beta_0+\]
\[+p_{v_1}\beta_1+p_{v_2}\beta_2+\cdots=\lambda(u_0\beta+u_1\gamma)\]
where $c=\gamma(Z)=\alpha_0(Z)=z^0.$ The conditions
\begin{equation}\label{e3.10} 2u_1z+p_{u_1}=\lambda u_0,\ 2u_0v_1z+p_{u_0}=\lambda u_1 \end{equation}
determine $z$ and the useless factor $\lambda$ in terms of function $p.$ Function $p$ is subjected to the remaining conditions
\[-2u_0c+u_0p_{u_1}+p_{v_0}=0, p_{v_1}=p_{v_2}=\cdots =0.\]
It follows that $p=p(x,u_0,u_1,v_0)$ and inserting (\ref{e3.8}) for $c,$ only one differential equation
\begin{equation}\label{e3.11} u_0^2(p_x+u_1p_{u_0})+2u_1^2(p_{v_0}+u_1p_{u_1})=2u_0u_1p \end{equation}
for this function appears. Some particular solutions can be explicitly found.
\begin{remark}\label{rem4}
In fact we have obtained all the infinitesimal symmetries $Z$ and the reasons are as follows. In one direction, one can observe that only the forms $\omega=\lambda\pi_0$ $(\lambda\neq 0)$ have the property that the family $\omega,\mathcal L_D\omega,\mathcal L^2_D\omega,\ldots\,$ generates the diffiety~$\Omega.$ Since every infinitesimal symmetry $Z$ preserves this property, it does satisfy~(\ref{e3.9}). In the opposite direction, (\ref{e3.9}) means that the vector field $Z$ preserves the Pfaffian equation $\pi_0=0$ and therefore preserves the \emph{adjoint space} to this equation. This is a~finite--dimensional space and therefore $Z$ generates a~Lie group. (We recall that the adjoint space consists of the most economical family of functions such that the equation $\pi_0=0$ can be expressed in terms of these functions. In our case, three functions
\[\frac{u_0}{u_1},\ u_0v_0-u_1\ln{u_0},\ 2x-u_0u_1\]
are enough. It follows that vector fields $Z$ generate the Lie contact group preserving the equation $\pi_0=0$ in three--dimensional underlying space.)
\end{remark}
%%%%%%%%%%%%%%%%%%
\subsection{On the evolutionary operators}\label{ssec3.8}
Assuming $z=0$ in equations (\ref{e3.10}), then the resulting conditions $p_{u_1}=\lambda u_0,p_{u_0}=\lambda u_1$ imply $p=p(x,u_0u_1,v_0)$ and equation (\ref{e3.11}) admits only trivial solution $p=0$ of this kind. It follows that the evolutional operators generating a~group do not exist.
%%%%%%%%%%%%%%%%%%%%%%%%%%%%%%%%%%%%%%%%%%%%%%%%
\section{Several independent variables}\label{sec4}
For the partial differential equations, the simplified requirements on the variations can be found as well, however, they remain complicated so that no general existence or non--existence results are actually available. The finite algorithm can be found only for the higher--order symmetries of (roughly saying) $m$ equations with $m+1$ unknown functions at most. The differential equation
\[\frac{\partial u}{\partial x_n}=F(x_1,\ldots,x_n,u,v,\frac{\partial u}{\partial x_1},\ldots,\frac{\partial u}{\partial x_{n-1}},\frac{\partial v}{\partial x_1},\ldots,\frac{\partial v}{\partial x_n})\]
where $u=u(x_1,\ldots,x_n), v=v(x_1,\ldots,x_n)$ may serve for an~example. We however suppose $n=2$ for technical reasons.

Let us introduce space $\mathbf M$ with coordinates
\[x,y,u_r,v_{rs}\quad (r,s=0,1,\ldots\,),\]
diffiety $\Omega$ with the basis
\[\alpha_r=\text{d}u_r-u_{r+1}\text{d}x-D^r_yF\text{d}y, \beta_{rs}=\text{d}v_{rs}-v_{r+1,s}\text{d}x-v_{r,s+1}\text{d}y\quad (r,s=0,1,\ldots\,)\]
where $F=F(x,y,u_0,v_{00},u_1,v_{10},v_{01})$ and the total derivatives
\[\begin{array}{l}
D_x=\frac{\partial}{\partial x}+\sum u_{r+1}\frac{\partial}{\partial u_r}+\sum v_{r+1,s}\frac{\partial}{\partial v_{rs}},\\\\
D_y=\frac{\partial}{\partial y}+\sum D^r_xF\frac{\partial}{\partial u_r}+\sum v_{r,s+1}\frac{\partial}{\partial v_{rs}}.
\end{array}\]
Identities $\mathcal L_{D_x}\alpha_r=\alpha_{r+1},$
\begin{equation}\label{e4.1}
\mathcal L_{D_y}\alpha_0=F_{u_0}\alpha_0+F_{v_{00}}\beta_{00}+F_{u_1}\alpha_1+F_{v_{10}}\beta_{10}+F_{v_{01}}\beta_{01}
\end{equation}
and
\[\mathcal L_{D_y}\alpha_r=\mathcal L_{D_y}\mathcal L^r_{D_x}\alpha_0=\mathcal L^r_{D_x}\mathcal L_{D_y}\alpha_0,
\ \mathcal L_{D_x}\beta_{rs}=\beta_{r+1,s},\ \mathcal L_{D_y}\beta_{rs}=\beta_{r,s+1}\]
are obvious.

In order to simplify the notation, we abbreviate
\[u=u_0, v=v_{00}, v_x=v_{10}, v_y=v_{01}, \alpha=\alpha_0,\beta=\beta_{00},\beta_x=\beta_{10},\beta_y=\beta_{01}\]
from now on.
%%%%%%%%%%%%%%%%%%
\subsection{Towards the variations}\label{ssec4.1}
We suppose
\[Z=z^1\frac{\partial}{\partial x}+z^2\frac{\partial}{\partial y}+\sum z_r\frac{\partial}{\partial\alpha_r}+\sum z_{rs}\frac{\partial}{\partial\beta_{rs}}\quad (z_r=\alpha_r(Z), z_{rs}=\beta_{rs}(Z))\]
in the direct approach. Identity (\ref{e4.1}) immediately gives the requirement
\[D_yz_0=F_uz_0+F_vz_{00}+F_{u_1}D_xz_0+F_{v_x}D_xz_{00}+F_{v_y}D_yz_{00}\]
for the initial terms $z_0=\alpha(Z)$ and $z_{00}=\beta(Z)$ of the recurrences
\[z_{r+1}=D_xz_r,\ z_{r+1,s}=D_xz_{rs},\ z_{r,s+1}=D_yz_{rs}.\]
Passing to the better approach, identity (\ref{e4.1}) is simplified as
\begin{equation}\label{e4.2}
\mathcal L_{D_y}\gamma=F_u\gamma+A\beta+B\beta_x+F_{u_1}\mathcal L_{D_x}\gamma\quad(\gamma=\alpha-F_{v_y}\beta)
\end{equation}
where
\[A=F_v+F_uF_{v_y}-D_yF_{v_y}+F_{u_1}D_xF_{v_y},\ B=F_{v_x}+F_{u_1}F_{v_y}.\]
The forms $\alpha_r$ can be replaced with $\gamma_r=\mathcal L^r_{D_x}\gamma$ $(r=0,1,\ldots\,)$ in the original basis of $\Omega.$ So we obtain
\begin{equation}\label{e4.3}
Z=z^1\frac{\partial}{\partial x}+z^2\frac{\partial}{\partial y}+\sum c_r\frac{\partial}{\partial\gamma_r}+\sum z_{rs}\frac{\partial}{\partial\beta_{rs}}\quad (c_r=\gamma_r(Z)) \end{equation}
together with the recurrence $c_{r+1}=D_xc_r$ $(r=0,1,\ldots\,)$ and the requirement
\begin{equation}\label{e4.4} D_yc_0=F_uc_0+Az_{00}+BD_xz_{00}+F_{v_x}D_xc_0 \end{equation}
for the initial terms $c_0$ and $z_{00}.$ Alas, the existence of a~solution is still ambigous. Explicit solution can be found if $A\neq 0$ but $B=0$ identically.
%%%%%%%%%%%%%%%%%%
\subsection{Infinitesimal symmetries}\label{ssec4.2}
Variations $Z$ such that moreover
\begin{equation}\label{e4.5}
\mathcal L_Z\beta=\lambda^1\beta+\lambda^2\gamma,\ \mathcal L_Z\gamma=\mu^1\beta+\mu^2\gamma
\end{equation}
are just the infinitesimal symmetries. In more detail, let us abbreviate $b=z_{00}=\beta(Z), c=c_0=\gamma(Z).$ Then
\[\mathcal L_Z\beta=Z\rfloor\text{d}\beta+\text{d}b,\ \mathcal L_Z\gamma=Z\rfloor\text{d}\gamma+\text{d}c\]
should be inserted into (\ref{e4.5}) with
\[\begin{array}{ll}
\text{d}\beta &=\text{d}x\wedge\beta_x+\text{d}y\wedge\beta_y,\\
\text{d}\gamma &=\text{d}\alpha-F_{v_y}\text{d}\beta-\text{d}F_{v_y}\wedge\beta\\
&=\text{d}x\wedge\gamma_1+\text{d}y\wedge\mathcal L_{D_y}\gamma-(F_{v_yu}\alpha+F_{v_yu_1}\alpha_1+F_{v_yv_x}\beta_x+F_{v_yv_y}\beta_y)\wedge\beta
\end{array}\]
and
\[\text{d}b\cong b_{u_1}\alpha_1+b_{v_x}\beta_x+b_{v_y}\beta_y+(\cdot),\ 
\text{d}c\cong c_{u_1}\alpha_1+c_{v_x}\beta_x+c_{v_y}\beta_y+(\cdot) \]
(mod $\beta,\gamma)$ where $(\cdot)$ denotes the higher--order summands. The final result follows by a~mere routine. By looking at the isolated terms $(\cdot),$ one can infer that $b,c$ depend only on the lower--order variables $x,y,u,v,u_1,v_x,v_y.$ So the requirement (\ref{e4.4}) is reduced to a~finite dimension. Moreover
\[z^1+b_{v_x}=z^2+b_{v_y}=b_{u_1}=0\]
easily follows from the first equation (\ref{e4.5}) and the second equation implies
\begin{equation}\label{e4.6}\begin{array}{c}
bF_{v_yv_y}+c_{v_y}=0,\\
z^1+z^2F_{u_1}+bF_{v_yu_1}+c_{u_1}=0,\\
-z^1F_{v_y}+z^2F_{u_x}+bF_{v_yv_x}+c_{v_x}=0.
\end{array}\end{equation}
It follows that the symmetries are subjected to very strong additional requirements, however, they are of the classical finite--dimensional nature. Assuming $z^1=z^2=0,$ then $Z=\frak Z$ becomes the evolutional operator and the system (\ref{e4.6}) is compatible: the Frobenius theorem comfortably applies.
\begin{remark}\label{rem5}
Variations $Z$ preserving the Pfaffian system $\beta=\gamma=0$ preserve the adjoint variables and therefore generate a~group. The converse assertion will be mentioned in broader context later on.
\end{remark}
%%%%%%%%%%%%%%%%%%%%%%%%%%%%%%%%%%%%%%%%%%%%%%%%
\section{Trivial differential equations}\label{sec5} We recall the space $\mathbf M=\mathbf M(m,n)$ with jet coordinates (\ref{e1.4a}), the diffiety $\Omega=\Omega(m,n)$ of the forms (\ref{e1.4b}) and the total derivatives $D_i$ in (\ref{e1.3}). They correspond to the trivial (empty) system (\ref{e1.5a}). The variations
\begin{equation}\label{e5.1} Z=\sum z_iD_i+\sum z^j_I\frac{\partial}{\partial w^j_I}\qquad (z_i=Zx_i,\,z^j_I=\omega^j_I(Z)) \end{equation}
are given by the formulae
\begin{equation}\label{e5.2} z^j_I=\omega^j_I(Z)=D_I\omega^j(Z)=D_Iz^j\quad (I=i_1\cdots i_r,\, D_I=D_{i_1}\cdots D_{i_r}) \end{equation}
already stated in (\ref{e1.5}).

We recall the triviality
\begin{equation}\label{e5.3}
\text{d}f=\sum\frac{\partial f}{\partial x_i}\text{d}x_i+\sum\frac{\partial f}{\partial w^j_I}\text{d}w^j_I=\sum D_if\text{d}x_i+\sum\frac{\partial f}{\partial\omega^j_I}\omega^j_I\quad (f\in\mathcal F)
\end{equation}
which implies the equalities $\partial f/\partial w^j_I=\partial f/\partial\omega^j_I.$ The inclusion $\mathcal L_Z\Omega\subset\Omega$ is clearly equivalent to $\mathcal L_Z\mathcal H\subset\mathcal H$ whence the rules
\begin{equation}\label{e5.4}\begin{array}{c}
\mathcal L_ZD_i=[Z,D_i]=-D_iz_i\cdot D_i,\ \mathcal L_Z\mathcal L_{D_i}-\mathcal L_{D_i}\mathcal L_Z=-D_iz_i\cdot\mathcal L_{D_i},\\
\mathcal L_Z\omega^j_{Ii}=\mathcal L_Z\mathcal L_{D_i}\omega^j_I=\mathcal L_{D_i}\mathcal L_Z\omega^j_I-D_iz_i\omega^j_{Ii}
\end{array}\end{equation}
easily follow. If $\Omega(m,n)_l\subset\Omega(m,n)$ is the submodule generated by all forms $\omega^j_I$ $(j=1,\ldots,m;\,I=i_1\cdots i_r;\,|I|=r\leq l)$ then $\Omega_*:\Omega(m,n)_0\subset\Omega(m,n)_1\subset\cdots\subset\Omega(m,n)$ is a~good "order--preserving" filtration.
\begin{theorem}[Lie--B\"acklund]\label{t5.1}
Let $m>1$ and $Z$ be a~variation such that $\mathcal L_Z\Omega(m,n)_l\subset\Omega(m,n)_l$ for appropriate $l.$ Then the inclusion is satisfied for all $l=0,1,\ldots\,$ and $Z$ is the infinitesimal point symmetry in the sense that all functions $Zx_i, Zw^j$ $(i=1,\ldots,n;\,j=1,\ldots,m\,)$ depend only on coordinates $x_1,\ldots,x_n,w^1,\ldots,w^m.$
\end{theorem}
\begin{proof}
Abbreviating $\Omega_l=\Omega(m,n)_l$ for now, we assume $\mathcal L_Z\Omega_l\subset\Omega_l$ with a~certain $l.$ Then $\mathcal L_Z\Omega_{l+1}\subset\Omega_{l+1}$ follows by applying (\ref{e5.4}) to the equality $\Omega_{l+1}=\Omega_l+\mathcal L_\mathcal H\Omega_l.$ Let $l>0$ and $\omega\in\Omega_{l-1}.$ Then $\mathcal L_{D_i}\omega\in\Omega_l$ hence $\mathcal L_Z\mathcal L_{D_i}\omega\in\Omega_l$ and (\ref{e5.4}) implies  $\mathcal L_{D_i}\mathcal L_Z\omega\in\Omega_l$ whence $\mathcal L_Z\omega\in\Omega_{l-1}$ and $\mathcal L_Z\Omega_{l-1}\subset\Omega_{l-1}.$ So we may assume $\mathcal L_Z\Omega_0\subset\Omega_0.$ In more detail
\[\mathcal L_Z\omega^j=Z\rfloor\text{d}\omega^j+\text{d}\,\omega^j(Z)=\sum\lambda^j_k\omega^k\qquad (j=1,\ldots,m)\]
where $\text{d}\omega^j=\sum\text{d}x_i\wedge\omega^j_i.$ Applying moreover (\ref{e5.3}) to the functions $f=\omega^j(Z)=z^j,$ one can directly obtain the identities
\[z_i+\sum\frac{\partial z^j}{\partial w^j_i}=0, \frac{\partial z^j}{\partial w^k_i}=0\ (j\neq k),\] \[\frac{\partial z^j}{\partial w^k_I}=0 \ (i=1,\ldots,n;\,j,k=1,\ldots,m;\,|I|>1).\]
Since $m>1,$ we infer that
\[z_i=z_i(\cdot\cdot,x_{i'},w^{j'},\cdot\cdot),\ z^j=-\sum z_iw^j_i+F^j(\cdot\cdot,x_{i'},w^{j'},\cdot\cdot)\]
for appropriate functions $F^j$ and therefore
\[Zw^j=\text{d}w^j(Z)=\omega^j(Z)+\sum w^j_i\text{d}x_i(Z)=z^j+\sum w^j_iz_i=F^j.\]
This concludes the proof. \end{proof}
\begin{remark}\label{rem6} The original Lie--B\"acklund theorem concerns the \emph{symmetries of the finite--order jet spaces}. It was rigorously proved much later and we refer to the extensive discussion \cite[pp. 66--80]{T8}. Alas, this rather tedious method fails for the \emph{diffieties} where we refer to short tricky proof \cite{T3}. \end{remark}

The above result for the \emph{infinitesimal symmetries} can be completed as follows.
\begin{theorem}\label{th5.2} All infinitesimal symmetries of diffieties $\Omega(1,n)$ are only the classical Lie contact transformations.\end{theorem}
\begin{proof}
Lemma~\ref{l2.2} can be applied with module $\Gamma=\Omega(1,n)$ which is generated by single form $\omega^1.$ It follows that all forms \[\omega^1,\mathcal L_Z\omega^1,\mathcal L_Z^2\omega^1,\ldots\] must be of a~limited order and this is possible if and only if $\mathcal L_Z\omega^1=\lambda\omega^1$. \end{proof}
\begin{remark}\label{rem6a} We note on this occasion that the classical "wave mechanism" of Lie's contact transformations of diffiety $\Omega(1,n)$ can be  generalized to obtain many \emph{higher--order symmetries of diffieties $\Omega(m,n)$ where $m>1$} \cite{T9,T10}. They destroy the order--preserving filtration $\Omega(m,n)_*.$ \end{remark}
Let us conclude with the \emph{infinitesimal symmetries} $Z$ such that
\begin{equation}\label{e5.5}\mathcal L_Z^2\Omega(m,n)_0\subset\Omega(m,n)_0+\mathcal L_Z\Omega(m,n)_0\subset\Omega(m,n)_1. \end{equation}
Alas, even this "simple" symmetry problem cannot be resolved in full generality here. So we suppose $m=2$ and the more explicit condition
\begin{equation}\label{e5.6} \mathcal L_Z\pi=\lambda\pi,\ \mathcal L_Z\omega^2=\sum\mu^j\omega^j+\sum\lambda_i\mathcal L_{D_i}\pi\qquad (\pi=\omega^1+a\omega^2) \end{equation}
where $j=1,2$ and $a$ is a~given function. Requirement (\ref{e5.3}) is satisfied since
\[\mathcal L_Z\mathcal L_{D_i}\pi=\mathcal L_{D_i}\mathcal L_Z\pi-D_iz_i\mathcal L_{D_i}\pi=D_i\lambda\pi+(\lambda-D_iz_i)\mathcal L_{D_i}\pi\in\Omega(2,n)_1\]
by applying (\ref{e5.4}).

Passing to the symmetry problem, the obvious identities
\[\begin{array}{ll}
\mathcal L_Z\pi=\mathcal L_Z\omega^1+a\mathcal L_Z\omega^2+Za\,\omega^2,&\mathcal L_{D_i}\pi=\omega^1_i+a\omega^2_i+D_ia\,\omega^2\\
\mathcal L_Z\omega^j=Z\rfloor\sum\text{d}x_i\wedge\omega^j_i+\text{d}z^j,& \text{d}z^j=\sum D_iz^j\text{d}x_i+\sum\frac{\partial z^j}{\partial\omega^k_I}\omega^k_I
\end{array}\]
directly give the following result. The coefficients $z^j=z^j(\cdot\cdot,x_{i'},w^{j'},w^{j'}_{i'},\cdot\cdot)$ are of the first order at most and satisfy the conditions
\[\mathcal Dz^1+a\mathcal Dz^2=Za,\ \mathcal D_iz^1=a\mathcal D_iz^2,\ z_i=\mathcal D_iz^2\]
where
\[\mathcal D=a\frac{\partial}{\partial w^1}-\frac{\partial}{\partial w^2},\ \mathcal D_i=a\frac{\partial}{\partial w^1_i}-\frac{\partial}{\partial w^2_i},\ Za=\sum z_iD_ia+\sum D_Iz^j\frac{\partial a}{\partial w^j_I}.\]

Explicit solution with $\lambda_i\neq 0$ could be stated in the particular case when $a=a(x_1,\ldots,x_n).$ In any case, we conclude that the higher--order symmetries of trivial equations are highly nontrivial.
%%%%%%%%%%%%%%%%%%%%%%%%%%%%%%%%%%%%%%%%%%%%%%%%
\section{Further perspectives}\label{sec6}
We have employed only elementary tools as yet. The general theory however rests on more advanced principles.
%%%%%%%%%%%%%%%%%%
\subsection{The Lie--Poisson bracket}\label{ssec6.1}
A~variation $Z$ of diffiety $\Omega$ is determined if the initial terms of certain recurrences, namely some functions $\omega(Z)$ with special $\omega\in\Omega,$ are known. It follows that the familiar identity
\begin{equation}\begin{array}{rl}\label{e6.1}
\omega([X,Y])&=X\omega(Y)-Y\omega(X)-\text{d}\omega\,(X,Y)\\
&=\mathcal L_Y\omega(X)-\mathcal L_X\omega(Y)+\text{d}\omega\,(X,Y) \end{array}\end{equation}
determines the bracket $[X,Y]$ in terms of variations $X,Y.$ If in particular $X=\frak X,$ $Y=\frak Y$ are the evolutional symmetries (hence $\frak Xx_i=\frak Yx_i=0$) and $\omega=\omega_f=\text{d}f-\sum D_if\text{d}x_i$ is a~contact form, we obtain the triviality
\[\begin{array}{rl}
[\frak X,\frak Y]f&=\omega_f([\frak X,\frak Y])=\frak X\omega_f(\frak Y)-\frak Y\omega_f(\frak X)\\
&=\omega_{\frak Yf}(\frak X)-\omega_{\frak Xf}(\frak Y)=\frak X\frak Yf-\frak Y\frak Xf.\end{array}\]
However, nontrivial interpretation is possible \cite{T8}. Denoting
\begin{equation}\label{e6.2} \{F,G\}=\frak XF-\frak YG\qquad (F=\omega_f(\frak X), G=\omega_f(\frak Y)),
\end{equation}
we have the \emph{higher--order Poisson bracket $\{\cdot\}.$} The sense of the construction lies in the schema
\[\frak X\leftrightarrow\omega_f(\frak X),\ \frak Y\leftrightarrow\omega_f(\frak Y),\ [\frak X,\frak Y]\leftrightarrow\{\omega_f(\frak X),\omega_f(\frak Y)\}\]
which means that the Lie bracket of vector fields corresponds to the Poisson bracket of (families of) certain functions. Using (\ref{e6.1}), analogous "bracket" can be introduced for general variations $X,Y$ and forms $\omega\in\Omega$ as well.
%%%%%%%%%%%%%%%%%%
\subsection{The role of involutivity \cite{T5,T11}}\label{ssec6.2} Let $\Omega_l\subset\Omega$ be a~term of a~good filtration $\Omega_*$ of a~diffiety~$\Omega$ and $x_1,\ldots,x_n$ be "not too special" independent variables. We introduce the following construction.

1: Let $\omega^1,\ldots,\omega^{\sigma_1}\in\Omega_l$ be a~maximal family such that $$\mathcal L_{D_1}\omega^1,\ldots,\mathcal L_{D_1}\omega^{\sigma_1}$$ are linearly independent forms mod $\Omega_l.$

2: Let $\omega^{\sigma_1+1},\ldots,\omega^{\sigma_1+\sigma_2}\in\Omega_l$ be a~maximal family such that $$\mathcal L_{D_2}\omega^{\sigma_1+1}, \ldots, \mathcal L_{D_2}\omega^{\sigma_1+\sigma_2}$$ are linearly independent forms mod $\Omega_l$ and the previous $\mathcal L_{D_1}\omega^1,\ldots,\mathcal L_{D_1}\omega^{\sigma_1}.$

$\vdots$

$n-1$: Let $\omega^{\sigma_1+\cdots\sigma_{n-2}+1},\ldots,\omega^{\sigma_1+\cdots\sigma_{n-1}}\in\Omega_l$ be a~maximal family such that  $$\mathcal L_{D_{n-1}}\omega^{\sigma_1+\cdots\sigma_{n-2}+1},\ldots,\mathcal L_{D_{n-1}}\omega^{\sigma_1+\cdots\sigma_{n-1}}$$ are linearly independent forms mod $\Omega_l$ and the previous $\mathcal L_{D_1}\omega^1,\ldots,\mathcal L_{D_{n-2}}\omega^{\sigma_1+\cdots+\sigma_{n-2}}.$

$n\hspace{0,6cm}$: Let $\omega^{\sigma_1+\cdots\sigma_{n-1}+1},\ldots,\omega^{\sigma_1+\cdots\sigma_n}\in\Omega_l$ be a~maximal family such that  $$\mathcal L_{D_{n}}\omega^{\sigma_1+\cdots\sigma_{n-1}+1},\ldots,\mathcal L_{D_{n}}\omega^{\sigma_1+\cdots\sigma_{n}}$$ are linearly independent forms mod $\Omega_l$ and the previous $\mathcal L_{D_1}\omega^1,\ldots,\mathcal L_{D_{n-1}}\omega^{\sigma_1+\cdots+\sigma_{n-1}}.$

The \emph{involutivity theorem} reads: \emph{denoting}
\[\omega^j_{r_1\cdots r_n}=\mathcal L_{D_1}^{r_1}\cdots\mathcal L_{D_n}^{r_n}\omega^j,\]
\emph{and assuming $l$ large enough, then the forms}
\begin{equation}\label{e6.3}\begin{array}{l}
\omega^j_{r_1\cdots r_n}\ (j=1,\ldots,\sigma_1),\quad \omega^j_{0r_2\cdots r_n}\ (j=\sigma_1+1,\ldots,\sigma_2),\quad\ldots \\
\omega^j_{0\cdots 0r_{n-1}r_n}\ (j=\sigma_{n-2}+1,\ldots,\sigma_{n-1}),\quad\omega^j_{0\cdots 0r_n}\ (j=\sigma_{n-1}+1,\ldots,\sigma_{n})\end{array}\end{equation}
\emph{where $r_1,\ldots,r_n=1,2,\ldots\,$ are linearly independent.} It follows that together with a~basis of~$\Omega_l$ they provide a~basis of total diffiety $\Omega.$ The result is of a~fundamental importance. Variations $Z$ are determined by certain values $\omega(Z)$ where $\omega\in\Omega.$ If $\omega\in\Omega_l,$ these functions $\omega(Z)$ are subjected to a~certain \emph{finite number} of requirements arising from the differential equations under consideration. Assuming $\omega\notin \Omega_l,$ it is sufficient to consider  only forms (\ref{e6.3}) of the basis and \emph{they are subjected only to the recurrences}
\[\omega^j_{r_1+1,r_2\cdots r_n}(Z)=D_1\omega^j_{r_1,\cdots r_n}(Z),\ldots ,\ \omega^j_{r_1\cdots r_{n-1},r_n+1}(Z)=D_n\omega^j_{r_1,\cdots r_n}(Z).\] It should be however noted that this deep achievement does not much affect the earthly practice of routine calculations.
%%%%%%%%%%%%%%%%%%
\subsection{Adjustment of ordinary differential equations}\label{ssec6.3} 
The case of one independent variable and one total derivative $D=D_1$ is simple \cite{T4,T5}. If $\Omega_*$ is a~good filtration, the forms $\omega\in\Omega$ are organized into several ramified strings.\\
\begin{picture}(140,130)
\put(10,105){\line(1,0){10}}\put(20,105){\line(0,-1){80}}\put(10,110){$\Omega_0$}
\put(40,105){\line(1,0){10}}\put(50,105){\line(0,-1){80}}\put(40,110){$\Omega_1$}
\put(70,105){\line(1,0){10}}\put(80,105){\line(0,-1){80}}\put(70,110){$\Omega_2$}
\qbezier(10,95)(20,105)(34,95)\put(35,95){\vector(1,-1){1}}%4.patro
\qbezier(35,95)(45,105)(59,95)\put(59,95){\vector(1,-1){1}}\put(59,95){\circle{8}}
\qbezier(60,95)(70,105)(84,95)\put(84,95){\vector(1,-1){1}}\put(87,93){$\cdots$}
\qbezier(10,75)(20,85)(34,75)\put(35,75){\vector(1,-1){1}}%3.patro
\qbezier(35,75)(45,90)(59,94)\put(59,94){\vector(2,1){1}}
\qbezier(10,55)(20,65)(34,55)\put(35,55){\vector(1,-1){1}}%2.patro
\put(4,48){$\omega$}\put(27,43){$\mathcal L_D\omega$}
\put(54,43){$\mathcal L^2_D\omega$}
\qbezier(35,55)(45,65)(59,55)\put(59,55){\vector(1,-1){1}}\put(35,55){\circle{8}}
\qbezier(60,55)(70,65)(84,55)\put(84,55){\vector(1,-1){1}}\put(87,53){$\cdots$}
\qbezier(10,35)(20,45)(34,54)\put(35,54){\vector(2,1){1}}%1.patro
\put(20,10){Figure 1a)}
\end{picture}
\begin{picture}(170,130)
\put(10,105){\line(1,0){10}}\multiput(20,0)(0,5){16}{\put(0,25){\line(0,1){2}}}
\put(40,105){\line(1,0){10}}\multiput(50,0)(0,5){16}{\put(0,25){\line(0,1){2}}}
\put(70,105){\line(1,0){10}}\multiput(80,0)(0,5){16}{\put(0,25){\line(0,1){2}}}
\put(95,105){\line(1,0){10}}\put(105,105){\line(0,-1){80}}\put(95,110){$\Omega_1$}
\put(120,105){\line(1,0){10}}\put(130,105){\line(0,-1){80}}\put(120,110){$\Omega_2$}
\put(-3,86){$\pi^1_0$}
\qbezier(10,90)(20,100)(34,90)\put(35,90){\vector(1,-1){1}}%4.patro
\qbezier(35,90)(45,100)(59,90)\put(59,90){\vector(1,-1){1}}
\qbezier(60,90)(70,100)(84,90)\put(84,90){\vector(1,-1){1}}
\qbezier(85,90)(95,100)(109,90)\put(109,90){\vector(1,-1){1}}
\qbezier(110,90)(120,100)(134,90)\put(134,90){\vector(1,-1){1}}\put(140,88){$\cdots$}
\put(23,53){$\pi^2_0$}
\qbezier(35,55)(45,65)(59,55)\put(59,55){\vector(1,-1){1}}
\qbezier(60,55)(70,65)(84,55)\put(84,55){\vector(1,-1){1}}
\qbezier(85,55)(95,65)(109,55)\put(109,55){\vector(1,-1){1}}
\qbezier(110,55)(120,65)(134,55)\put(134,55){\vector(1,-1){1}}\put(140,53){$\cdots$}
\put(5,25){$\mathcal R$}
\qbezier(5,60)(-10,35)(5,35)\qbezier(5,35)(15,35)(10,60)\put(10,61){\vector(-1,3){1}}
\put(50,10){Figure 1b)}
\end{picture}\\
After an~appropriate change of lower--order terms, the strings become rectified and this provides the standard basis. If there is only one string, it does not change after any symmetry \cite[Theorem 26]{T4}. We have omitted some "noncontrollability" submodules $\mathcal R\subset\Omega$ of a~finite dimension here.
%%%%%%%%%%%%%%%%%%
\subsection{Adjustment of partial differential equations}\label{ssec6.4} The arrangements are not so easy (see \cite{T5}) and we can mention only the case of two independent variables. If $\Omega_*$ is a~good filtration, the forms $\omega\in\Omega$ are organized into two--dimensional sheets which may be ramified. The sheets can be made flat and each of them then consists of a~finite family of overlapping infinite "triangles". We present the case of only one ramified sheet and the corresponding adjustment with two overlapping triangles. In this particular case, the symmetries $Z$ preserve the vertices $\beta, \gamma.$\\
\begin{picture}(140,130)(0,-10)
\qbezier(10,85)(25,95)(40,90)\put(40,90){\vector(2,-1){1}}%1.patro
\qbezier(40,90)(55,100)(75,90)\put(75,90){\vector(2,-1){1}}
\qbezier(75,90)(90,100)(110,90)\put(110,90){\vector(2,-1){1}}\put(115,87){$\cdots$}
\qbezier(10,85)(25,65)(40,60)\put(40,60){\vector(2,-1){1}}
\qbezier(40,90)(55,70)(75,65)\put(75,65){\vector(1,0){1}}
\qbezier(75,90)(85,70)(107,70)\put(107,70){\vector(1,0){1}}\put(115,67){$\cdots$}
\qbezier(40,60)(55,70)(75,65)\put(75,65){\vector(1,0){1}}
\qbezier(75,65)(90,72)(107,70)
\qbezier(40,60)(55,45)(70,45)\put(74,45){\circle{8}}\put(104,50){\circle{8}}
\qbezier(78,45)(90,52)(100,50)
\qbezier(75,65)(90,52)(100,51)
\qbezier(75,41)(85,29)(100,27)\put(100,26){\vector(2,-1){1}}
\qbezier(106,47)(120,34)(130,32)\put(130,31){\vector(2,-1){1}}
\qbezier(106,53)(120,57)(130,55)\put(130,55){\vector(1,0){1}}
\put(15,95){$\mathcal L_{D_1}$}
\put(10,55){$\mathcal L_{D_2}$}
\qbezier[40](10,75)(25,85)(40,80)\put(40,80){\vector(2,-1){1}}%2.patro
\qbezier[40](40,80)(55,90)(75,80)\put(75,80){\vector(2,-1){1}}
\qbezier[40](75,80)(90,90)(110,80)\put(110,80){\vector(2,-1){1}}\put(115,77){$\cdots$}
\qbezier[40](10,75)(25,55)(40,50)\put(40,50){\vector(2,-1){1}}
\qbezier[40](40,80)(55,60)(75,55)\put(75,55){\vector(1,0){1}}
\qbezier[40](75,80)(85,60)(107,60)\put(107,60){\vector(1,0){1}}\put(115,57){$\cdots$}
\qbezier[30](40,50)(55,60)(75,55)\put(75,55){\vector(1,0){1}}
\qbezier[30](75,55)(90,62)(107,60)
\qbezier[40](40,50)(55,35)(70,43)
\qbezier[30](75,55)(90,43)(101,47)
\put(50,0){Figure 2a)}
\end{picture}
\begin{picture}(150,130)(0,-10)
\put(52,67){$\beta$}\put(56,70){\circle{13}}
\put(40,47){$\gamma$}\put(43,49){\circle{13}}
\qbezier(25,31)(75,15)(75,93)\qbezier(25,30)(25,105)(75,93)
\qbezier(45,55)(40,65)(32,74)\put(33,73){\vector(-1,1){1}}
\qbezier(56,77)(56,87)(42,88)\put(43,88){\vector(-2,0){1}}
\put(36,73){$\mathcal L_Z$}
\qbezier(48,45)(52,32)(65,30)\qbezier(48,45)(52,48)(70,45)
\qbezier(65,30)(70,18)(85,15)\qbezier(65,30)(70,35)(85,30)\put(85,15){\vector(2,-1){1}}
\qbezier(70,45)(70,40)(85,30)\qbezier(70,45)(77,50)(90,50)
\put(92,50){\circle{4}}
\qbezier(85,30)(90,22)(100,17)\put(85,15){\vector(2,-1){1}}\put(100,17){\vector(2,-1){1}}
\qbezier(85,30)(90,35)(108,35)\put(110,35){\circle{4}}
\put(112,35){\vector(2,1){10}}\put(112,35){\vector(2,-1){10}}
\qbezier(93,48)(97,44)(108,37)
\qbezier(93,52)(100,55)(115,55)\put(117,55){\circle{4}}
\put(119,55){\vector(2,1){10}}\put(119,55){\vector(2,-1){10}}
\qbezier(90,51)(75,51)(62,67)\qbezier(62,67)(80,77)(90,75)\put(90,75){\vector(1,0){1}}
\qbezier(90,75)(100,60)(115,56)\qbezier(90,75)(100,80)(115,76)\put(115,76){\vector(2,-1){1}}\put(120,74){$\cdots$}
\put(105,105){$\mathcal L_{D_1}\mathcal L_{D_2}=\mathcal L_{D_2}\mathcal L_{D_1}$}
\put(60,0){Figure 2b)}
\end{picture}\\
We omit the proof which rests on Lemma~\ref{l2.2} and also the mention of the separated "noncontrability" submodules $\mathcal R\subset\Omega$ consisting of strings.
%%%%%%%%%%%%%%%%%%
\subsection{Isospectral solutions and solitons}\label{ssec6.5} 
We briefly reinterpret the calculation \cite{T12} of the KdV--hierarchy in order to point out some general aspects. The original problem reads: let
\begin{equation}\label{e6.4}v_{xx}+(\lambda+q(x,t))v=0\quad (\lambda\in\mathbb R, v=v(x,t))
\end{equation}
be the eigenvalue problem depending on a~parameter $t$ and our task is to determine such evolution equations
\begin{equation}\label{e6.5}v_t=P\quad (P=A(\lambda;\cdot\cdot,q_r,\cdot\cdot)v+B(\lambda;\cdot\cdot,q_r,\cdot\cdot)v_x)\quad (q_r=\frac{\partial^rq}{\partial x^r})
\end{equation}
that the compatibility conditions of the system (\ref{e6.4}, \ref{e6.5}) are of the special kind
\begin{equation}\label{e6.6}q_t=Q(\cdot\cdot,q_r,\cdot\cdot).
\end{equation}
The eigenvalue $\lambda$ is preserved and does not affect the evolution of the function~$q.$

The reinterpretation is as follows. We start with the \emph{ordinary} differential equation
\[\frac{d^2v}{dx^2}+(\lambda+q)v=0\quad (\lambda\in\mathbb R, v=v(x),q=q(x))\]
with two unknown functions. In terms of diffieties, we have space $\mathbf M$ with coordinates $x,\lambda, v, v_x,q_0,q_1,\ldots\,$ and the module $\Omega$ with the basis
\[\text{d}\lambda, \alpha=\text{d}v-v_x\text{d}x, \alpha_x=\text{d}v_x+(\lambda+q_0)v\text{d}x, \beta_r=\text{d}q_r
-q_{r+1}\text{d}x\ (r=0,1,\ldots\,).\]
Our task is to determine variations $Z$ such that
\begin{equation}\label{e6.7}Zx=0\ (Z=\frak Z \text{ is the evolutional operator}),\ Z\lambda=0,\ Zq=Q(\cdot\cdot,q_r,\cdot\cdot).\end{equation}
Roughly saying, function $q$ is "autonomous" in the evolution. Quite general variations $Z$ can be determined with the use of the standard basis
\[\pi_0=\alpha, \pi_1=\mathcal L_D\pi_0=\alpha_1, \pi_2=\mathcal L_D\pi_1=-(\lambda+q)\alpha-v(\text{d}\lambda+\beta_0), \ldots\] where
\[D=\frac{\partial}{\partial x}+v_x\frac{\partial}{\partial v}-(\lambda+q)\frac{\partial}{\partial v_x}+\sum q_{r+1}\frac{\partial}{\partial q_r}\] is the total derivative. Assuming moreover $Zx=Z\lambda=0,$ we obtain
\[Z=\sum D^rP\frac{\partial}{\partial \pi_r}\quad (P=\pi_0(Z)=\alpha(Z)=Zv)\]
where $P$ may be an~arbitrary function. Then
\[D^2P=\pi_2(Z)=-(\lambda+q)Zv-vZq=-(\lambda+q)P-vQ\]
and the last requirement (\ref{e6.7}) is expressed by
\begin{equation}\label{e6.8}Q(\cdot\cdot,q_r,\cdot\cdot)=-\frac{1}{v}(D^2P+(\lambda+q)P) \end{equation}
which provides very strong additional conditions on the function $P.$ First of all, we infer that
\[P_{vv}=P_{vv_x}=P_{v_xv_x}=0\]
by looking at the higher--order term $D^2P$ in (\ref{e6.8}). So we may suppose
\[P=A(\lambda;\cdot\cdot,q_r,\cdot\cdot)v+B(\lambda;\cdot\cdot,q_r,\cdot\cdot)v_x+C(\lambda;\cdot\cdot,q_r,\cdot\cdot).\] Substitution into (\ref{e6.8}) then yields
\[2\mathcal DA+\mathcal D^2B=0,\ C=0\]
and finally
\begin{equation}\label{e6.9}Q(\cdot\cdot,q_r,\cdot\cdot)=\frac{1}{2}\mathcal D^3B+2(\lambda+q_0)\mathcal DB+q_1B\quad (\mathcal D=\sum q_{r+1}\frac{\partial}{\partial q_r}).
\end{equation}
It is not easy to discuss this equation in full generality, however, there are particular solutions
\[B=B_0(\cdot\cdot,q_r,\cdot\cdot)\lambda^n+\cdots +B_{n-1}(\cdot\cdot,q_r,\cdot\cdot)\lambda+B_n(\cdot\cdot,q_r,\cdot\cdot)\] for any $n=0,1,\ldots\,.$ We recall the final result \cite{T12}
\[B_0=1, B_1=-\frac{1}{2}q_0, B_2=\frac{1}{8}(q_2+3q_0^2), \ldots\]
which moreover provides the famous KdV--hierarchy
% $Zq=Q=$
\[Zq=Q=(q_1=)\,q_x, (q_2+3q_0^2)_x, (q_4+5q_1^2+10qq_2+10q_0^3)_{xx}, \ldots\]
as the final result.

The reinterpretation of the problem indicates some delicate features of the theory which were not yet discussed in actual literature and incorporate the original problem in much broather context. Analogous approach can be applied to nonlinear and partial differential equations.
%%%%%%%%%%%%%%%%%%%%%%%%%%%%%%%%%%%%%%%%

\end{document}